\documentclass{amsart}
\usepackage[T1]{fontenc}
\usepackage{amsmath,amsthm,amssymb,amsfonts,enumerate,mathrsfs}
\newtheorem{theorem}{Theorem}
\newtheorem{lemma}{Lemma}

\theoremstyle{definition}
\newtheorem{definition}{Definition}

\newtheorem{remark}{Remark}

\newtheorem{question}{Question}

\newcommand{\N}{\mathbb{N}}

\title{The Primes Are $2$-Accessible}
\author{Oscar Quester}
\date{\today}

\begin{document}

\begin{abstract}

We prove that the set of positive integers having between $1$ and $n$ prime factors, counted with multiplicity, has degree of accessibility $2^n$. In particular, the case $n=1$ answers a question of Landman and Robertson asking whether the set of prime numbers is $2$-accessible.

\end{abstract}
\maketitle

\section{Introduction}

An \textit{$r$-coloring} of $\mathbb{N}$ is a function from $\mathbb{N}$ to a set with $r$ elements, usually $\{1,\dots,r\}$. A set $A \subseteq \mathbb{N}$ is said to be \textit{monochromatic} with respect to a coloring $\chi$ if $\chi$ is constant on $A$, that is, if $|\chi(A)|=1$. In this paper, we study the Ramsey properties of sequences with prescribed gaps.

\begin{definition}

Let $D \subseteq \N$. A $k$-term sequence of positive integers $$x_1,x_2,\dots,x_k$$ with $x_{i+1}-x_i \in D$ for all $1 \leq i \leq k-1$ is called a $k$-term \textit{$D$-diffsequence}.

\end{definition}

\begin{definition}
    
A set $D \subseteq \N$ is called \textit{$r$-accessible} if, for every $r$-coloring of $\N$, there exist arbitrarily long monochromatic sequences $x_1,x_2,\dots,x_k$ with $x_{i+1}-x_i \in D$; that is, there exist monochromatic $k$-term $D$-diffsequences for every $k \geq 1$. If $D$ is $r$-accessible for every $r \in \N$, we say that $D$ is \textit{accessible}.
    
\end{definition}

Given a set $D \subseteq \N$, we define the \textit{degree of accessibility} of $D$, denoted $\operatorname{doa}(D)$, to be the greatest positive integer $r$ for which $D$ is $r$-accessible. If $D$ is $r$-accessible for every $r \in \N$, we write $\operatorname{doa}(D)=\infty$. The study of accessibility was initiated by Landman and Robertson \cite{landman2007avoiding}. This property has since been studied further in \cite{ardal2008ramsey,clifton2024new,farhangi2021distance,jungic2005conjecture,landman2014ramsey,landmanramsey,landman2010avoiding,oscar2025,wesley2022improved}.

For $D \subseteq \mathbb{N}$ and a positive integer $c$, let $$ D+c=\{d+c : d\in D\} \quad \text{and} \quad D-c=\{d-c : d \in D,\ d>c\}. $$ Equivalently, $D-c=\{n \in \mathbb{N} :n+c \in D\}$. Let $\mathbb{P}=\{2,3,5,7,\dots\}$ denote the set of prime numbers. For $D \subseteq \mathbb{N}$, let $$D^{(n)}=\{d_1\cdots d_k : 1 \leq k \leq n \text{ and } d_1,\dots,d_k \in D\}.$$ Thus, $$ \mathbb{P}^{(n)} = \{p_1 \cdots p_k : 1 \leq k \leq n \text{ and } p_1,\dots,p_k \in \mathbb{P}\}, $$ that is, $\mathbb{P}^{(n)}$ is the set of positive integers having between $1$ and $n$ prime factors, counted with multiplicity.

A basic periodic $3$-coloring of $\mathbb{N}$ shows that the primes are not $3$-accessible \cite{landman2007avoiding}. In addition, the authors of \cite{landman2007avoiding} prove that $\mathbb{P}+c$ is $2$-accessible for every odd $c \in \mathbb{N}$. Question $7$ of \cite{landman2007avoiding}, along with Research Question $10.2$ of \cite{landman2014ramsey}, asks whether $\mathbb{P}$ is $2$-accessible. The $2$-accessibility of the primes is also studied in \cite{landman2010avoiding}, where the authors note that every periodic $2$-coloring of $\mathbb{N}$ contains arbitrarily long monochromatic $\mathbb{P}$-diffsequences. Using a Szemerédi-type theorem for the shifted primes $\mathbb{P}-1$ and $\mathbb{P}+1$ due to Frantzikinakis, Host, and Kra \cite{frantzikinakis2013polynomial}, we answer this question affirmatively.

\begin{theorem} \label{thm:D}

Let $N \geq 2$. Then $$\operatorname{doa}\left(\bigcup_{t=1}^{N-1}t\mathbb{P}\right)=N.$$

\end{theorem}

The fact that the primes are $2$-accessible follows immediately by letting $N=2$. In addition, we obtain a slightly more general result.

\begin{theorem}\label{thm:primes}

For every $n \geq 1$, we have $$\operatorname{doa}(\mathbb{P}^{(n)})=2^n.$$ In particular, when $n=1$, we obtain $$\operatorname{doa}(\mathbb{P})=2,$$ that is, the primes are $2$-accessible.
    
\end{theorem}

\begin{definition}

We say that $A \subseteq \mathbb{N}$ has \textit{positive upper Banach density} if $$d^*(A)=\limsup_{|I| \to \infty}\frac{|A \cap I|}{|I|}>0,$$ where the limsup is taken over all intervals $I \subseteq \mathbb{N}$. We say that $E \subseteq \mathbb{N}$ is a set of \textit{density recurrence} if, for every set $A \subseteq \mathbb{N}$ with $d^*(A)>0$, $$E \cap \{n \in \mathbb{N} : d^*\left(A \cap (A-n) \cap \cdots \cap (A-kn)\right)>0\} \neq \varnothing$$ for every $k \geq 1$.
    
\end{definition}

Frantzikinakis, Host, and Kra proved a polynomial Szemerédi-type theorem for the shifted primes $\mathbb{P}-1$ and $\mathbb{P}+1$, which implies the following.

\begin{theorem}[\cite{frantzikinakis2013polynomial}, Theorem 1.1]\label{thm:frant}
The shifted primes $\mathbb{P}-1$ and $\mathbb{P}+1$ are sets of density recurrence. 
\end{theorem}

\begin{remark} 

In fact, Frantzikinakis, Host, and Kra proved a much more general result. Namely, if $\ell,m \in \mathbb{N}$ and $\vec{q}_1,\dots,\vec{q}_m:\mathbb{Z}\to \mathbb{Z}^{\ell}$ 
are polynomials with $\vec{q}_i(0)=\vec{0}$ for $i=1,\dots,m$, then for any $E\subseteq \mathbb{Z}^{\ell}$ with positive upper Banach density, 
that is, $d^*(E)>0$, the set of integers $n$ such that
$$
d^*\left(
E\cap (E-\vec{q}_1(n))\cap \cdots \cap (E-\vec{q}_m(n))
\right)>0
$$
has nonempty intersection with both $\mathbb{P}-1$ and $\mathbb{P}+1$.
    
\end{remark}

The following simple lemma, which is an iterative application of the definition of density recurrence, will be useful.

\begin{lemma}\label{lem:density-grid}

Let $E_1,\dots,E_n \subseteq \mathbb{N}$ be sets of density recurrence. Then for any set $A$ of positive upper Banach density and every $k \geq 1$, there exist $a \in A$ and $e_i \in E_i$ for $1 \leq i \leq n$ such that $$ \{a+i_1e_1+\cdots+i_ne_n : 0 \leq i_j \leq k \text{ for } 1 \leq j \leq n\} \subseteq A. $$

\end{lemma}

\begin{proof}

Since $E_1$ is a set of density recurrence and $d^*(A)>0$, $$E_1 \cap\{n \in \mathbb{N} : d^*(A \cap (A-n) \cap \cdots \cap (A-kn))>0\} \neq \varnothing.$$ Let $e_1 \in E_1$ be such that $$d^*(A \cap (A-e_1) \cap \cdots \cap (A-ke_1))>0$$ and let $A_1=A \cap (A -e_1) \cap \cdots \cap (A-ke_1)$. Since $E_2$ is a set of density recurrence and $d^*(A_1)>0$, $$E_2 \cap \{n \in \mathbb{N} : d^*(A_1 \cap (A_1-n) \cap \cdots \cap (A_1-kn))>0\} \neq \varnothing.$$ Let $e_2 \in E_2$ be such that $$d^*(A_1 \cap (A_1-e_2) \cap \cdots \cap (A_1-ke_2))>0$$ and let $$A_2=A_1 \cap (A_1-e_2) \cap \cdots \cap (A_1-ke_2).$$ We continue this process inductively to obtain $$A_n=A_{n-1} \cap (A_{n-1}-e_n) \cap \cdots \cap (A_{n-1}-ke_n)$$ with $d^*(A_n)>0$. By construction, $$A_n=\bigcap_{0 \leq i_1,\dots,i_n \leq k}(A-i_1e_1-i_2e_2 - \cdots - i_ne_n).$$ Since $d^*(A_n)>0$, choose $a \in A_n$. Then $$ \{a+i_1e_1+i_2e_2+\cdots+i_ne_n : 0 \leq i_j \leq k \text{ for } 1 \leq j \leq n\} \subseteq A, $$ as desired. \end{proof}

We are now ready to prove the main result, Theorem \ref{thm:D}. Before giving the proof, we briefly describe the idea. Given a coloring, we look at the set $X$ of starting points of monochromatic pairs at distances less than $N$. If $X$ has arbitrarily long gaps, then a simple pigeonhole argument forces long monochromatic arithmetic progressions with common difference $N$. If $X$ has bounded gaps, then one of the sets $$X_{\ell,c}=\{m \in \mathbb{N} : \chi(m)=\chi(m+\ell)=c\}$$ has positive upper Banach density. The recurrence of $\mathbb{P}-1$ and $\mathbb{P}+1$ then gives a two-dimensional grid inside $X_{\ell,c}$, from which we extract a long monochromatic diffsequence.

\begin{proof}[Proof of Theorem \ref{thm:D}]

Let $$D_N=\bigcup_{t=1}^{N-1}t\mathbb{P}.$$ Let $\chi : \mathbb{N} \to \{1,\dots,N\}$ be an arbitrary $N$-coloring of $\mathbb{N}$. Define $$X_{\ell,c}=\{n \in \mathbb{N} : \chi(n)=\chi(n+\ell)=c\}$$ for $\ell \in \{1,\dots,N-1\}$ and $c \in \{1,\dots,N\}$. Let $$X=\bigcup_{\ell=1}^{N-1}\bigcup_{c=1}^{N}X_{\ell,c}.$$ We consider two cases.

\medskip 
\noindent \textbf{Case 1.} Suppose $X$ has unbounded gaps. Then, for every $k \geq 1$, there exists an interval $[a,a+kN]$ of integers such that $$X \cap [a,a+kN]=\varnothing.$$ Now, we claim $$\chi(a+(j-1)N)=\chi(a+jN)$$ for all $j \in \{1,\dots,k\}$. Indeed, suppose the endpoint colors are different. Among the $N+1$ integers $$a+(j-1)N,\ a+(j-1)N+1,\dots,\ a+jN$$ two must have the same color. Since the endpoints have different colors, these two equal-colored points are at distance between $1$ and $N-1$. Thus the smaller of them belongs to $X$, contradicting $$X\cap [a,a+kN]=\varnothing.$$ Hence, $$\chi(a)=\chi(a+N)=\cdots=\chi(a+kN).$$ Since $N \in D_N$ (if $p \mid N$, then $N=(N/p)p$ and $N/p\leq N-1$), we have a $(k+1)$-term monochromatic $D_N$-diffsequence. Since $k \geq 1$ was arbitrary, we conclude that if $X$ has unbounded gaps, then $\chi$ admits arbitrarily long monochromatic $D_{N}$-diffsequences.

\medskip
\noindent \textbf{Case 2.} Suppose $X$ has bounded gaps. Hence $X$ is syndetic and thus has positive upper Banach density. Since $X$ is a finite union of the sets $X_{\ell,c}$, there exist $\ell \in \{1,\dots,N-1\}$ and $c \in \{1,\dots,N\}$ such that the set $$X_{\ell,c}=\{m \in \mathbb{N}:\chi(m)=\chi(m+\ell)=c\}$$ has positive upper Banach density. We apply Lemma \ref{lem:density-grid} with $$A=X_{\ell,c}, \quad E_1=\mathbb{P}-1, \quad E_2=\mathbb{P}+1,$$ to conclude that for any $k \geq 1$, there exist $a \in X_{\ell,c}$ and primes $p,q \in \mathbb{P}$ such that $$\{a+i_1(p-1)+i_2(q+1) : 0 \leq i_1,i_2 \leq k\ell \} \subseteq X_{\ell,c}.$$ By definition of $X_{\ell,c}$, this implies $$\chi\left(a+i_1(p-1)+i_2(q+1)+\epsilon\right)=c
$$ for all $0\leq i_1,i_2\leq k\ell$ and $\epsilon\in\{0,\ell\}$. Taking $i_1=i_2=i \ell$ and $\epsilon=0$, we find $$\chi(a+i\ell(p+q) )=c$$ for all $0 \leq i \leq k$. Taking $i_1=(i+1)\ell$, $i_2=i\ell$, and $\epsilon=\ell$, we find $$\chi(a+(i+1)\ell p+i\ell q)=c$$ for all $0 \leq i \leq k-1$. Now let $$x_{2i+1}=a+i\ell(p+q) \quad \text{for } 0 \leq i \leq k$$ and $$x_{2i+2}=a+(i+1)\ell p + i\ell q \quad \text{for } 0 \leq i \leq k-1.$$ Then $$x_1<x_2<\cdots<x_{2k+1}$$ is a $(2k+1)$-term monochromatic sequence with $$x_{2i+2}-x_{2i+1}=\ell p \quad \text{and} \quad x_{2i+3}-x_{2i+2}=\ell q.$$ Since $\ell\leq N-1$, both $\ell p$ and $\ell q$ belong to $D_N$. Hence, we have a $(2k+1)$-term monochromatic $D_{N}$-diffsequence. Since $k \geq 1$ was arbitrary, we conclude that if $X$ has bounded gaps, then $\chi$ admits arbitrarily long monochromatic $D_{N}$-diffsequences. Since $\chi$ was an arbitrary $N$-coloring, we conclude that $D_{N}$ is $N$-accessible, that is, $$\operatorname{doa}(D_{N}) \geq N.$$

We now prove that $D_N$ is not $(N+1)$-accessible. Define $$B_c=\{c+N(c(N-1)+r):0\leq r\leq N-2\}$$ for $c \in \{0,\dots,N-1\}$. Let $$B=\bigcup_{c=0}^{N-1}B_c$$ and let $s=\max B.$ Note that the absolute difference between distinct elements of the same $B_c$ is at least $N$, and the absolute difference between elements of distinct $B_c$'s is at least $N+1$. Let $T=N^2s!$. Since $B\subseteq [0,s]$ and $T=N^2s!>2s$, even after reducing modulo $T$, every nonzero difference of two elements of $B$ still has absolute value at least $N$. We define an $(N+1)$-coloring $\chi : \mathbb{N} \to \{0,\dots,N\}$ by $$\chi(x)=\begin{cases}
    x\bmod N & : x\bmod T \notin B, \\
    N & : x\bmod T \in B.
\end{cases}$$ Suppose $$x_1<x_2<\cdots<x_k$$ is a monochromatic $D_N$-diffsequence under $\chi$. Suppose $\chi(x_1)=c$ for some $c \in \{0,\dots,N-1\}$. Then $N \mid x_{i+1}-x_i$. For each $i$, write $x_{i+1}-x_i=t_i p_i$ with $1\leq t_i\leq N-1$ and $p_i\in\mathbb{P}$. Since the sequence is monochromatic of color $c$, we have $N\mid t_i p_i$. Since $t_i<N$, this forces $p_i\mid N$, and in particular $p_i\leq N$. Hence the positive multiple $x_{i+1}-x_i$ of $N$ is at most $N(N-1)$. Therefore $$x_{i+1}-x_i \in \{N,2N,\dots,(N-1)N\}.$$ Put $y_i=(x_i-c)/N$. Then $y_1<\cdots<y_k$ is a monochromatic $\{1,2,\dots,N-1\}$-diffsequence for the induced coloring $y\mapsto \chi(c+Ny)$, which has period $T/N$. In each period, the $N-1$ consecutive residues
$$c(N-1),c(N-1)+1,\dots,c(N-1)+(N-2)$$ have color $N$. Since $T/N \geq N$, a color-$c$ $\{1,2,\dots,N-1\}$-diffsequence cannot enter or cross one of these blocks, so its length is bounded. Now suppose $\chi(x_1)=N$. For each $i$, let $s_i=x_i\bmod T$, with residues taken in $\{0,1,\dots,T-1\}$. Then $s_i\in B$. Write $$x_{i+1}-x_i=t_ip_i,$$ where $1\leq t_i\leq N-1$ and $p_i\in\mathbb P$. Then $$s_{i+1}-s_{i}\equiv x_{i+1}-x_i\equiv t_ip_i\pmod T.$$ We first claim that $s_{i+1}-s_{i}\neq 0$. If $s_{i+1}-s_{i}=0$, then $T\mid t_ip_i$. Since $t_i\leq N-1\leq s$, we have $t_i\mid s!$, and hence $t_i\mid T$. Therefore $$\frac{T}{t_i}\mid p_i.$$ This is impossible since $T/t_i=N^2s!/t_i>1$ is composite, and hence $s_{i+1}-s_{i}\neq 0$. By construction, every nonzero difference between two elements of $B$ has absolute value at least $N$. Thus $$N\leq |s_{i+1}-s_{i}|\leq s.$$ Hence $|s_{i+1}-s_{i}|\mid s!$, so $|s_{i+1}-s_{i}|\mid T$. Since $s_{i+1}-s_i\equiv t_ip_i\pmod T$, it follows that $|s_{i+1}-s_i|\mid t_ip_i$. Since $|s_{i+1}-s_i|>t_i$, we must have $p_i\mid |s_{i+1}-s_i|$, and therefore $$\left|\frac{t_ip_i}{s_{i+1}-s_{i}}\right|\leq t_i\leq N-1.$$ Write $$t_ip_i=s_{i+1}-s_{i}+m_iT$$ for some $m_i\in\mathbb Z$. If $m_i\neq 0$, then $$ \left|\frac{t_ip_i}{s_{i+1}-s_{i}}\right| = \left|1+m_i\frac{T}{s_{i+1}-s_{i}}\right| \geq \frac{T}{|s_{i+1}-s_{i}|}-1. $$ Since $|s_{i+1}-s_{i}|\leq s$, we have $$ \frac{T}{|s_{i+1}-s_{i}|}-1\geq \frac{T}{s}-1>N-1, $$ contradicting the previous bound. Hence $m_i=0$, so $$s_{i+1}-s_{i}=t_ip_i>0.$$ Thus $$s_1<s_2<\cdots<s_k.$$ Since each $s_i$ lies in the finite set $B$, we have $k\leq |B|$. Hence there are no arbitrarily long monochromatic $D_N$-diffsequences of color $N$. Together with the boundedness already proved for the colors $0,\dots,N-1$, this shows that $\chi$ does not admit arbitrarily long monochromatic $D_N$-diffsequences. Hence $D_N$ is not $(N+1)$-accessible. Combining this with our lower bound, we find $$\operatorname{doa}\left(\bigcup_{t=1}^{N-1}t\mathbb{P}\right)=N,$$ as desired. \end{proof}

With Theorem \ref{thm:D}, Theorem \ref{thm:primes} follows fairly quickly.

\begin{proof}[Proof of Theorem \ref{thm:primes}]

Note that $$\bigcup_{t=1}^{2^n-1}t\mathbb{P}\subseteq \mathbb P^{(n)}.$$ For the upper bound, consider the following $(2^n+1)$-coloring $\chi : \mathbb{N} \to \{0,\dots,2^n\}$ defined by $$\chi(x)=\begin{cases}
    x\bmod 2^n & : 3^{n+1} \nmid x, \\
    2^n & : 3^{n+1} \mid x.
\end{cases}$$ This coloring is a generalization of the $3$-coloring used by Landman and Robertson in \cite{landman2007avoiding} to prove the primes are not $3$-accessible. We claim that there are no monochromatic $3^{n+1}$-term $\mathbb{P}^{(n)}$-diffsequences under this coloring. Indeed, suppose $$x_1<x_2<\cdots<x_{3^{n+1}}$$ is such a monochromatic $\mathbb{P}^{(n)}$-diffsequence. We cannot have $\chi(x_1)=2^n$ or else $3^{n+1} \mid x_{2}-x_1$, but $x_{2}-x_1 \in \mathbb{P}^{(n)}$, which contains no multiple of $3^{n+1}$. Finally, if $\chi(x_1)=c$ for some $c \in \{0,\dots,2^n-1\}$, then $$x_{i+1}-x_{i} \equiv 0 \bmod{2^n}$$ for all $1 \leq i \leq 3^{n+1}-1$. Since $2^n$ is the only multiple of $2^n$ in $\mathbb{P}^{(n)}$, it follows that $$x_{i}=x_1+(i-1)2^n.$$ Since $\gcd(2^n,3^{n+1})=1$, the residues of $$x_1,x_2,\dots,x_{3^{n+1}}$$ form a complete residue system modulo $3^{n+1}$, and hence one term has color $2^n$, a contradiction. Hence, $\chi$ admits no monochromatic $3^{n+1}$-term $\mathbb{P}^{(n)}$-diffsequences, so $\mathbb{P}^{(n)}$ is not $(2^n+1)$-accessible, and it follows that $$\operatorname{doa}(\mathbb{P}^{(n)})=2^n,$$ as desired. \end{proof}

\begin{question}[\cite{landman2007avoiding,landman2010avoiding}]

Landman and Robertson proved that $\mathbb{P}+c$ is $2$-accessible for every odd $c \in \mathbb{N}$. It would be nice to determine whether $\mathbb{P}+c$ is $2$-accessible for any even $c \in \mathbb{N}$. 
    
\end{question}

\begin{question}

Recall that for $D \subseteq \mathbb{N}$, we defined $$D^{(n)}=\{d_1\cdots d_k : 1 \leq k \leq n \text{ and } d_1,\dots,d_k \in D\}.$$ For which sets $D$ does one have $$\operatorname{doa}(D^{(n)})=(\operatorname{doa}(D))^n$$ for every $n\geq 1$?
    
\end{question}

\vskip20pt\noindent {\bf Acknowledgments.} At the beginning of this project, a large language model suggested considering the shifted-prime recurrence theorem of Frantzikinakis, Host, and Kra \cite{frantzikinakis2013polynomial} in connection with the Landman--Robertson question on the $2$-accessibility of the primes. A large language model was also used for suggestions on improving the exposition of the paper.

\end{document}